\documentclass[12pt, reqno]{amsart}

\usepackage{todonotes}
\usepackage[utf8]{inputenc}
\usepackage{lmodern}

\usepackage{comment}

\usepackage[scale=0.83,centering]{geometry}
\usepackage{amsmath,amsfonts,mathabx,amssymb,amsthm,graphicx,bbm}
\usepackage{color,mathrsfs,tikz,hyperref,extarrows}
\usepackage[font={footnotesize}]{caption}

\usetikzlibrary{arrows,automata}

\numberwithin{equation}{section}

\newcommand{\bbE}{{\ensuremath{\mathbbm E}} }

\newcommand{\bbN}{{\ensuremath{\mathbbm N}} }

\newcommand{\bbP}{{\ensuremath{\mathbbm P}} }

\newcommand{\bbR}{{\ensuremath{\mathbbm R}} }

\newcommand{\cH}{{\ensuremath{\mathcal H}} }

\newcommand{\cN}{{\ensuremath{\mathcal N}} }

\newcommand{\bE}{{\ensuremath{\mathbf E}} }

\newcommand{\bP}{{\ensuremath{\mathbf P}} }

\newcommand{\gep}{\varepsilon}       

\newcommand{\sL}{\mathsf{L}}

\newcommand{\what}{\widehat}

\newcommand{\ind}{\mathbf{1}}
\newcommand{\dd}{{\ensuremath{\mathrm d}}}
\newcommand{\cov}{\text{Cov}}
\newcommand{\var}{\text{Var}}

\newtheorem{theorem}{Theorem}[section]

\newtheorem{lemma}[theorem]{Lemma}

\newtheorem{proposition}[theorem]{Proposition}
\newtheorem{remark}{Remark}[section]

\newcommand{\red}{\color{red}}

\begin{document}

\title[Pinning model beyond the $L^2$-regime]{Scaling limit for the  pinning model in correlated Gaussian environment beyond the $L^2$-regime}

\author[J.\ Song]{Jian Song}\address{Research Center for Mathematics and Interdisciplinary Sciences, Shandong University, Qingdao, China} \email{txjsong@sdu.edu.cn}

\author[M.\ Wang]{Meng Wang}\address{Center for Applied Mathematics, Tianjin University, Tianjin, China}
\email{wmeng\_2206@tju.edu.cn}

\author[R.\ Wei]{Ran Wei}\address{Department of Financial and Actuarial Mathematics, Xi'an Jiaotong-Liverpool University, Suzhou 215123, China}
\email{ran.wei@xjtlu.edu.cn}

\date{}

\begin{abstract}
In this paper, we study the scaling limit of the  pinning model in correlated Gaussian environment. The tail probability of the underlying renewal process of the model has a polynomial decay with exponent $\alpha>0$. The covariance of the Gaussian environment $\{\omega_n\}_{n\in\bbN}$ is given by $\cov_{\bbP}(\omega_n,\omega_m)\sim 
|n-m|^{2H-2}$ with $H\in(0,1)$.  Assuming $\alpha\in(0,\frac12]$, $H\in(\frac12,1)$ and $\alpha+2H>2$, we show that the partition function of the disordered pinning model, under the appropriate scaling, converges in distribution to the $L^1$-solution of the fractional stochastic heat equation driven by  Gaussian noise correlated in time and localized  at the origin. In particular, it is known that the solution is not $L^2$-integrable when $\alpha<\frac12$.

\vspace{0.2cm}
\textit{Keywords}:
Disordered pinning model; stochastic heat equation; correlated Gaussian environment; scaling limit; Weinrib--Halperin prediction; $L^1$-solution
\end{abstract}

\maketitle

\tableofcontents

\section{Introduction}\label{sec:intro}
We study the scaling limit for the partition function of the disordered pinning model in correlated Gaussian environment. The disordered pinning model is a class of random polymer models. Despite its simple setup, the model exhibits rich behavior and is particularly well suited for studying phase transitions and critical phenomena. During the last two decades, the disordered pinning model has received widespread attention in mathematics, physics, chemistry and biology. We refer to \cite{giacomin:07:random,Gia11,hollander:09:random} for an overview and more details.

In the study of random polymers in random environments (also referred to as \emph{disorder}) such as the disordered pinning model, a central question is whether disorder is relevant, that is, whether an arbitrarily small amount of disorder fundamentally alters the long-term behavior of the model compared with the case without disorder. For correlated disorder, a criterion for determining the relevance of disorder was proposed in \cite{weinrib.halperin:83:critical}, known as the \textit{Weinrib--Halperin prediction}.

In this paper, we show that, under the appropriate conditions, the partition function of the disordered pinning model converges in distribution to the $L^1$-solution 
of a fractional stochastic heat equation (SHE) driven by colored noise, which confirms a disorder-relevant regime beyond the $L^2$-regime.

In the rest of this subsection, we first introduce the corresponding disordered pinning model and the fractional SHE in Section \ref{sec:model}. Then we introduce the research background and related studies in Section \ref{sec:review} and state our main result in Section \ref{sec:main_result}. Finally, we discuss our result and propose some open questions in Section \ref{sec:discuss}

\subsection{Models}\label{sec:model}
In this section, we introduce the disordered pinning model and the fractional SHE. We also heuristically illustrate the connection between the two models.

First, we introduce the disordered pinning model in correlated random environment.

Let $\tau:=\{\tau_0:=0<\tau_1<\tau_2<\cdots\}\subset\bbN\cup\{0\}$ be a renewal process with probability $\bP$ and expectation $\bE$, respectively. We assume that $\tau$ obeys the following inter-arrival law
\begin{equation}\label{def:inter_arr}
\bP(\tau_1=n)=\frac{L(n)}{n^{1+\alpha}},\quad\forall n\geq1,
\end{equation}
where $\alpha>0$ and $L(n):\bbR^+\to\bbR^+$ is a function slowly varying at the infinity (see \cite{BGT89} for more details of slowly varying functions).

Let $\omega:=\{\omega_n\}_{n\in\bbN}$ be a family of Gaussian random variables with probability $\bbP$ and expectation $\bbE$, respectively, representing the random environment. We assume $\omega$ is independent of $\tau$, {with
\begin{equation}\label{def:omega}
\bbE[\omega_n]=0\quad\text{and}\quad\gamma(n-m):=\cov_{\bbP}(\omega_n,\omega_m)\leq C\Big(|n-m|^{2H-2}\wedge 1\Big),\quad\forall n,m\in\bbN,
\end{equation}
for some positive constant $C$  and
\begin{equation}\label{eq:omega_asymp}
\lim\limits_{N\to\infty}N^{2-2H}\gamma(\lfloor tN\rfloor)=t^{2H-2},\quad\text{for any}~t>0,
\end{equation}
where} $H\in(0,1)$ is the \emph{Hurst parameter}. We shall discuss the assumption of Gaussianity in Section~\ref{sec:discuss}.

The disordered pinning model with length $N$ is defined via a Gibbs transform, by
\begin{equation}\label{def:pin}
\dd\bP_{N,\beta}^{\omega,h}(\tau):=\frac{1}{Z_{N,\beta}^{\omega,h}}e^{\sum_{n=1}^N(\beta\omega_n+h)\ind_{\{n\in\tau\}}}\dd\bP(\tau),
\end{equation}
where $\beta>0$ is the inverse temperature, $h\in\bbR$ is an external field, and
\begin{equation}\label{def:origin_pf}
Z_{N,\beta}^{\omega,h}:=\bE\Big[e^{\sum_{n=1}^N(\beta\omega_n+h)\ind_{\{n\in\tau\}}}\Big]
\end{equation}
is the partition function, which ensures that $\bP_{N,\beta}^{\omega,h}$ is a probability measure, called the \textit{polymer measure}.

In this paper, however, instead of the original partition function $Z_{N,\beta}^{\omega,h}$, we study a \textit{Wick-ordered} partition function, defined by
\begin{equation}\label{def:wick_pf}
\what{Z}_N^{\beta}:=\bE\Big[\exp\Big\{\sum\limits_{n=1}^N\beta\omega_n\ind_{\{n\in\tau\}}-\frac12\var_\bbP\Big(\sum\limits_{n=1}^N\beta\omega_n\ind_{\{n\in\tau\}}\Big)\Big\}\Big].
\end{equation}
The reason for this choice is that, since we are going to show that the scaling limit of the partition function converges to a mild Skorohod solution of a fractional SHE (see \eqref{def:she} below), it is natural to renormalize the partition function so that it becomes a Wick exponential. We also refer to \cite[Section 1.3]{LSWZ26} for a more detailed discussion.

Next, we introduce a fractional SHE on $\bbR$,  driven by a Gaussian noise $\dot W(t,x):=\delta_0(x)\xi(t)$ which is temporally correlated  and localized at the origin.  The SHE is defined by
\begin{equation}\label{def:she}
\begin{cases}
\displaystyle
\partial_t u(t,x)=-(-\Delta)^{\rho/2}u(t,x)+\beta u(t,x)\delta_0(x)\xi(t),\quad(t,x)\in\bbR^+\times\bbR,\\
u(0,x)\equiv1,
\end{cases}
\end{equation}
where $-(-\Delta)^{\rho/2}$ with $\rho\in(0,2]$ is the (fractional) Laplacian, $\delta_0(x)$ is the Dirac delta function at $x=0$, and $\xi$ is a Gaussian noise with covariance
\begin{equation}\label{def:cov_noise}
\bbE[\xi(t)\xi(s)]=|t-s|^{2H-2},\quad\forall t,s>0,
\end{equation}
where $H\in(0,1)$ is the same parameter as  in \eqref{def:omega}.

Note that here we slightly abuse the notation, denoting the probability and expectation for $\xi$ by $\bbP$ and $\bbE$, respectively, the same as those for the disorder $\omega$ in the pinning model. Moreover, $-(-\Delta)^{\rho/2}$ is the infinitesimal generator of the standard symmetric $\rho$-stable process $(X_t)_{t\geq0}$, and we denote the probability and expectation for $X_t$ by $\bP$ and $\bE$, respectively, the same as those for the renewal process $\tau$. In summary, $\bP$ and $\bE$ are for the underlying stochastic process of each model, and $\bbP$ and $\bbE$ are for the random environment of each model, and thus this abuse of the notation does not cause any ambiguity.

Now we illustrate the connection between the two models.

First, note that by the Feynman-Kac formula, the Skorohod solution of \eqref{def:she} is formally given by
\begin{equation}\label{def:feynman-kac-L2}
u(t,x)=\bE_{(t,x)}\Big[\exp\Big\{\beta\int_0^t\delta(X_s)\xi(t-s)\dd s-\frac{\beta^2}{2}\var_\bbP\Big(\int_0^t\delta(X_s)\xi(t-s)\dd s\Big)\Big\}\Big],
\end{equation}
where $\bE_{(t,x)}[\cdot]:=\bE[\cdot|X_t=x]$ denotes the expectation for the backward symmetric $\rho$-stable process $X$ on $[0,t]$ with terminal condition $X_t=x$.

On the other hand, if we replace $\ind_{\{n\in\tau\}}$ in \eqref{def:wick_pf} by $\ind_{\{S_n=0\}}$, where $S=(S_n)_{n\geq0}$ is in the domain of attraction of a $\rho$-stable law, then it is easy to observe that the Wick-ordered partition function~\eqref{def:wick_pf} is a natural discretization of \eqref{def:feynman-kac-L2}. To be more precise, let $\sigma$ be the renewal process generated by the return time to $0$ of $S$, then the inter-arrival law of $\sigma$ satisfies \eqref{def:inter_arr} with
\begin{equation}\label{eq:walk_renew}
\begin{cases}
\alpha=\frac{1}{\rho}-1\in[0,\infty), &\text{for}~\rho\in(0,1],\\
\alpha=1-\frac{1}{\rho}\in(0,\frac12], &\text{for}~\rho\in(1,2],
\end{cases}
\quad\text{(see \cite{DK11,Kes63}).}
\end{equation}

We emphasize that \eqref{def:feynman-kac-L2} is only a formal solution of \eqref{def:she}. We shall discuss more about the fractional SHE in Section \ref{sec:review} below.

\subsection{Disorder-relevance and the Weinrib--Halperin prediction}\label{sec:review}
As we mentioned at the beginning of Section \ref{sec:intro}, a particularly important problem in studying the disordered pinning model is to determine whether disorder is relevant. To be specific, in the \textit{disorder-relevant regime}, for any $\beta>0$, the behavior of $\tau$ under $\bP_{N,\beta}^{\omega,h=0}$ is essentially different from that under $\bP$. In contrast, in the \textit{disorder-irrelevant regime}, for small enough $\beta>0$, the features of the model under both probability measures $\bP_{N,\beta}^{\omega,h=0}$ and $\bP$ are comparable. In practice, since the polymer measure $\bP_{N,\beta}^{\omega,h}$ is quite sophisticated, we can alternatively study the change of some \textit{critical exponents} or \textit{critical points} to determine disorder-relevance.

For correlated disorder with $\cov_{\bbP}(\omega_n,\omega_m)=|n-m|^{-p}\wedge1$ for any $p>0$, \cite{weinrib.halperin:83:critical} predicted that disorder should be relevant if $\alpha>\frac12\min\{p,1\}$ and irrelevant if $\alpha<\frac12\min\{p,1\}$, where $\alpha$ is the exponent in \eqref{def:inter_arr}. In particular, if $p=1$, that is, $H=\frac12$ in \eqref{def:omega} and \eqref{def:cov_noise}, then since the fractional Brownian motion with Hurst parameter $H=\frac12$ is a standard Brownian motion, by convention, we refer this case to a model with i.i.d.\ disorder or white noise, instead of assuming $\gamma(n-m)\sim|n-m|^{-1}$. For the disordered pinning model in i.i.d.\ random environment, the Harris criterion \cite{Harris74} conjectured that disorder is relevant if $\alpha>\frac12$ and is irrelevant if $\alpha<\frac12$, which coincides with the Weinrib--Halperin prediction. We mention that the model in i.i.d. random environment has been intensively studied in recent years, including the marginal case $\alpha=\frac12$, which is inconclusive by the Harris criterion. A list of references includes, but is not limited to, e.g., \cite{Lac10,GT06,GTL10,BL18,CTT17,caravenna.sun.ea:17:polynomial,caravenna.sun.ea:17:universality}.

While the Harris criterion has been confirmed by a series of research, the study of Weinrib--Halperin prediction remains largely open. For the disordered pinning model in correlated Gaussian environments with $p>1$, \cite{Ber13} confirmed that disorder is relevant when $\alpha>\frac12$ by observing a change of critical exponents of the free energy, verifying the sufficient condition for disorder-relevance when the Weinrib--Halperin prediction coincides with the Harris criterion. However, when $p<1$, \cite{Ber13} showed that the quenched free energy is always positive and thus, there is no critical exponent. Hence, the classical approach, that is, by comparing the critical exponents to determine whether disorder is relevant, is no longer applicable. This phenomenon in the case $p<1$ was also studied in \cite{Ber14}, where it is called \textit{infinite disorder}. It was also observed in \cite{Ber13} that whether the covariances of the random environment are summable plays a crucial role. Later, \cite{BP15} extended the results for $p>1$ in \cite{Ber13}, only assuming the covariances are summable, and \cite{Poi13} obtained sharp aysmptotics for the annealed free energy by assuming the covariances are doubly summable of exponentially decayed. Very recently, by only assuming that the covariances are summable, \cite{GLZ26} showed that in the \textit{localized regime}, the quenched free energy is $C^\infty$ with respect to the external field $h$ by providing bounds for each order of the derivatives, and the number of contacts between the renewal processes and disorder satisfies a central limit theorem, where the method was first developed for the model of i.i.d.\ random environment in \cite{GZ24}. Specifically, if the random environments have only finite-range correlations, then it was shown in \cite{Poi12,Poi13+} that the Harris criterion still applies.


Note that the case $p>1$ covers the case $H\in(0,\frac12)$, while the case $p<1$ is equivalent to $H\in(\frac12,1)$. In the latter case, the Weinrib--Halperin prediction claims that disorder is relevant if $\alpha+H>1$ and is irrelevant if $\alpha+H<1$. As we mentioned above, when $H\in(\frac12,1)$, instead of classical approaches such as comparing the critical exponents, a new perspective is needed to determine whether disorder is relevant. In the disorder-relevant regime, it is possible to tune down the strength of disorder to zero along with time (i.e., send $\beta=\beta_N\to0$ as $N\to\infty$) in an appropriate rate, so that the partition function converges to some non-trivial random limit, in which disorder still takes effect. This approach, called the \textit{intermediate disorder regime}, was first developed for the directed polymer model in i.i.d.\ random environment in \cite{alberts.khanin.ea:14:intermediate}, and then it was extended to various disordered systems (including the disordered pinning model) in i.i.d.\ random environments in \cite{caravenna.sun.ea:17:polynomial}. In the intermediate disorder regime, the rescaled partition function converges in distribution to an $L^2$-integrable random variable, and in particular, it coincides with the disorder-relevant regime for i.i.d.\ random environments, even for the marginal case $\alpha=\frac12$ of the Harris criterion (see \cite{caravenna.sun.ea:17:universality}).

The intermediate disorder regime has been shown to be a robust framework. It has been applied to directed polymers in various correlated random environments, e.g., \cite{Rang20,RSW24,CG23,SSSX21}. Very recently, this approach was also applied to study the disordered pinning model with $H\in(\frac12,1)$ in \cite{LSWZ26}. It was shown that by choosing $\beta_N=\hat{\beta}L(N)N^{-(\alpha+H-1)}$, the rescaled partition function \eqref{def:wick_pf} converges in distribution to an $L^2$-integrable random variable, which can be expressed by a chaos expansion, when $\alpha>\frac12$ for any $\hat{\beta}>0$, and when $\alpha=\frac12$ for sufficiently small $\hat{\beta}>0$. Since in both cases, $\alpha+H>1$ holds automatically, this partially confirmed the disorder-relevant regime in the Weinrib--Halperin prediction. Moreover, it was also shown that each chaos in the above $L^2$-limit is well-defined when $\alpha+H>1$. However, when $\alpha<\frac12$, the $L^2$-norm of the Wick-ordered partition function diverges as $N\to\infty$, providing an evidence that the intermediate disorder regime cannot be applied in the entire disorder-relevant regime. This phenomena is new, compared to the models in i.i.d.\ random environments, where the $L^2$-intermediate disorder regime coincides with the disorder-relevant regime. It was also shown in \cite{LSWZ26} that the Wick-ordered partition function is uniformly integrable when $\alpha+2H>2$, without the assumption $\alpha\geq\frac12$. However, the limit of the Wick-ordered partition function was not identified because of the absence of certain useful structures, such as a martingale structure. Moreover, even if the existence of the limit could be established, it cannot be $L^2$-integrable when $\alpha<\frac12$.


Very recently, the seminal work \cite{QRV25} developed a theory of $L^1$-integrable Skorohod integral, and applied it to construct a global $L^1$-solution of the parabolic Anderson model on $\bbR^2$, which was known to have only a local $L^2$-solution (\cite{Hu02}). This approach provides a natural martingale structure, and inspired by it, \cite{li.song.ea:26:on} studied the fractional SHE \eqref{def:she} with $H\in(\frac12,1)$, which is a continuum counterpart of the disordered pinning model in \cite{LSWZ26}. Note that when $\rho=2$ (i.e., $\alpha=\frac12$ by \eqref{eq:walk_renew}) and $H=\frac12$, the noise $\xi$ is white in time and thus the SHE is a marginal case by Harris criterion. This special case was first studied in \cite{WY25}.


The mild solution of \eqref{def:she} is formally given by
\begin{equation}\label{def:mild}
\begin{split}
u(t,x)&=\int_\bbR g_\rho(t,x-y)u(0,y)\dd y+\beta\int_0^t\int_\bbR g_\rho(t-s,x-y)u(s,y)\delta_0(y)\dd y\xi(s)\dd s\\
&=1+\beta\int_0^t g_\rho(t-s,x)u(s,0)\xi(s)\dd s, 
\end{split}
\end{equation}
where $g_\rho(t,x)$ is the transition density of the $\rho$-stable process $X$ and the stochastic integral is in the sense of Skorohod. It was shown in \cite{li.song.ea:26:on} that when $\rho\in(0,2)$, if the solution of \eqref{def:mild} is unique, then it cannot be $L^p$-integrable for any $p>1$. In particular, the stochastic integral therein is not well-defined in the classical $L^2$-case. By the theory in \cite{QRV25}, the stochastic integral can be defined as an $L^1$-integrable random variable, and a global $L^1$-solution of \eqref{def:mild} was constructed explicitly in \cite{li.song.ea:26:on} when $\rho\in(\frac{1}{2H-1},2]$. Note that $\rho>1$ by $H\in(\frac12,1)$, and thus the corresponding disordered pinning model has the exponent $\alpha=1-\frac{1}{\rho}\in(0,\frac12]$. Hence, $\rho\in(\frac{1}{2H-1},2]\Longleftrightarrow\alpha+2H>2$ and $\alpha\leq\frac12$, which is consistent with the result in \cite{LSWZ26}. In particular, this result confirms the disorder-relevance regime for $\alpha+2H>2$ with $\alpha\leq\frac12$, which partially fills the gaps in the Weinrib--Halperin prediction.

In summary, we have the following up-to-date phase diagram \ref{Fig:phase-plane} for the disordered pinning model in correlated Gaussian environments. We shall further discuss the diagram in Section \ref{sec:discuss}.
\begin{figure}[ht]
\centering
\begin{tikzpicture}[x=1cm,y=1cm,font=\small]
\fill[blue!16]   (0,6) -- (4,6) -- (4,4.5) -- cycle;
\fill[orange!22] (0,6) -- (4,3) -- (4,4.5) -- cycle;
\fill[green!14]  (4,3) -- (8,3) -- (8,6) -- (4,6) -- cycle;
\fill[black!8]   (0,3) -- (4,3) -- (0,6) -- cycle;
\fill[black!5]   (0,0) -- (4,0) -- (4,3) -- (0,3) -- cycle; 
\fill[yellow!15]   (4,0) -- (8,0) -- (8,3) -- (4,3) -- cycle; 

\draw[densely dotted,black!45]
(0,3) -- (8,3);

\draw[->,thick] (-0.1,0) -- (8.7,0) node[right]{$\alpha$};
\draw[->,thick] (0,-0.1) -- (0,6.7) node[above]{$H$};
\draw (4,-0.08)--(4,0.08);   \node[below] at (4,-0.12){$\tfrac12$};
\draw (8,-0.08)--(8,0.08);   \node[below] at (8,-0.12){$1$};
\node[below] at (0,-0.12){$0$};
\draw (-0.08,3)--(0.08,3);   \node[left] at (-0.10,3){$\tfrac12$};
\draw (-0.08,4.5)--(0.08,4.5);\node[left] at (-0.10,4.5){$\tfrac34$};
\draw (-0.08,6)--(0.08,6);   \node[left] at (-0.10,6){$1$};

\draw[very thick,black!75,densely dashed] (0,3) -- (8,3);
\node[black!75,fill=white,inner sep=1pt,font=\footnotesize]
at (6.55,3) {$H=\tfrac12$: Harris criterion};

\draw[very thick] (0,6) -- (4,3);
\node[rotate=-38,anchor=south,font=\footnotesize] at (1.55,4.25) {$\alpha+H=1$};
\draw[very thick,densely dashdotted] (0,6) -- (4,4.5);
\node[rotate=-20,anchor=west,font=\footnotesize] at (1.8,5.5) {$\alpha+2H=2$};
\draw[very thick] (4,0) -- (4,3);
\draw[thick,densely dotted] (4,3) -- (4,6);

\fill (4,3) circle (2.4pt);
\node[anchor=north west,font=\footnotesize] at (4.08,2.95)
{$(\tfrac12,\tfrac12)$};

\node[align=center,font=\footnotesize] at (2.2,5.8)
{$L^1$-regime(solved)};
\node[align=center,font=\footnotesize] at (6.0,4.6)
{$L^2$-regime\\(solved)};
\node[rotate=-30,anchor=south,font=\footnotesize] at (2.2,4.4)
{relevant (open)};
\node[align=center,font=\footnotesize] at (1.0,3.6) {irrelevant\\(open)};
\node[align=center,font=\footnotesize] at (2.0,1.4)
{irrelevant\\ (open)};
\node[align=center,font=\footnotesize] at (6.0,1.4)
{relevant\\ (solved)};
\end{tikzpicture}
\caption{This is the phase diagram on the $(\alpha,H)$ plane. In \cite{Ber13}, it was shown that disorder is relevant in the yellow regime by observing a change of critical exponents, and the quenched free energy is always positive in the green regime. In \cite{LSWZ26}, it was shown that the green regime is an $L^2$-regime, and thus it is a disorder-relevant regime. In \cite{li.song.ea:26:on}, it was shown that the purple regime is an $L^1$-regime via the fractional SHE \eqref{def:she}, and thus it is also a disorder-relevant regime. By the Weinrib--Halperin prediction, the orange regime should be a disorder-relevant regime, and the two grey regions should be disorder-irrelevant regimes, but these are still open.}
\label{Fig:phase-plane}
\end{figure}
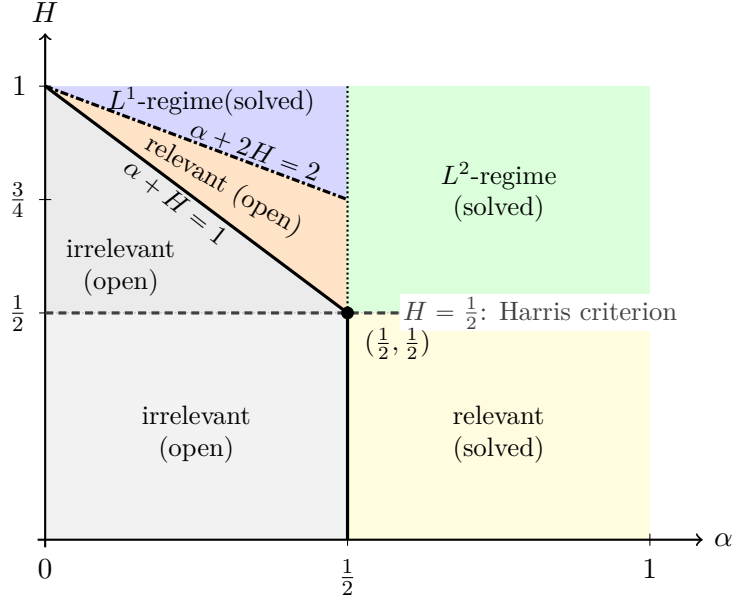

\subsection{Main result}\label{sec:main_result}
In Section \ref{sec:model}, we have explained that the Wick-ordered partition function \eqref{def:wick_pf} and the Feynman-Kac formula \eqref{def:feynman-kac-L2} are comparable. We have also seen that the same condition $\alpha+2H>2$ appeared in both \cite{LSWZ26} and \cite{li.song.ea:26:on}. Hence, it is natural to conjecture that by tuning $\beta_N$ appropriately, the Wick-ordered partition function converges to the $L^1$-solution of the fractional SHE~\eqref{def:she} when $\rho\in(\frac{1}{2H-1},2]\Longleftrightarrow\alpha+2H>2$ and $\alpha\leq\frac12$. By taking advantage of the explicit expression for the $L^1$-solution, which is given in Section \ref{sec:pre}, we provide an affirmative answer to this conjecture. The following theorem is our main result.


\begin{theorem}\label{thm:main}
Assume $H\in(\frac12,1)$ and $\rho\in(\frac{1}{2H-1},2]\Longleftrightarrow\alpha+2H>2$ and $\alpha\in(0,\frac12]$. Take
\begin{equation}\label{eq:scale_beta}
\beta_N:=\hat{\beta}\frac{L(N)}{N^{\alpha+H-1}}.
\end{equation}
Let $\what{Z}_{\lfloor Nt\rfloor}^{\beta_N}$ be defined as in \eqref{def:wick_pf}. Then as $N\to\infty$, 
\begin{equation}\label{eq:main_conv}
\what{Z}_{\lfloor Nt\rfloor}^{\beta_N}\overset{\text{(d)}}{\longrightarrow}u^{( \kappa_\alpha\hat{\beta})}(t,0),
\end{equation}
where $u^{(\kappa_\alpha\hat{\beta})}(t,x)$ is the $L^1$-solution of \eqref{def:mild} with $\beta$ replaced by $\kappa_\alpha\hat{\beta}$, and $\kappa_\alpha:=C_\alpha/g_{\rho}(1,0)$ with $C_\alpha$ the constant arising from the local limit theorem (see \cite{Don97})
\begin{equation}\label{eq:local_limit}
C_\alpha:=\lim\limits_{n\to\infty}n^{1-\alpha}L(n)\bP(n\in\tau)=\frac{\alpha\sin(\pi\alpha)}{\pi},
\end{equation}
and $g_\rho(t,x)$ the transition density of the symmetric stable process $X$.
\end{theorem}

{
\begin{remark}
We also mention that for $\rho=2\Longleftrightarrow\alpha=\frac12$, the convergence \eqref{eq:main_conv} in Theorem \ref{thm:main} can be proved in the $L^2$-regime for sufficiently small $t>0$. This is an immediate consequence of \cite[Theorem 1.1]{LSWZ26} and \cite[Proposition 3.2]{li.song.ea:26:on}, where we observe that the scaling limit of the Wick-ordered partition function coincides with the chaos expansion of the local $L^2$-solution of the SHE \eqref{def:she} .
\end{remark}
}

\subsection{Further discussion}\label{sec:discuss}
In this section, we discuss the setting and the result of our model, and some open problems.

\textbf{(Gaussian environment)} In this paper, we assume that the random environment is Gaussian. However, in view of related studies of disorder systems in independent environments, e.g., \cite{alberts.khanin.ea:14:intermediate,caravenna.sun.ea:17:polynomial,caravenna.sun.ea:17:universality}, it seems that this assumption is not necessary. To be more specific, a weaker assumption $\bbE[\exp(\beta\omega_1)]<+\infty$ for some $\beta>0$ should be sufficient, without assuming any particular distribution of $\omega$. Our approach relies on the assumption of Gaussianity due to technical reasons. In particular, the computations in Section \ref{sec:proof_prop1}, if it still works, could be very sophisticated without Gaussianity. Similar situations were also encountered in other studies of the same model, e.g.,\cite{Ber13,Poi13,LSWZ26}, where the random environments were all assumed to be Gaussian. We conjecture that the assumption of Gaussianity could be removed, but a very delicate treatment should be needed.

\textbf{(The condition $\alpha+2H>2$)} According to the Weinrib--Halperin prediction, disorder should be relevant when $\alpha+H>1$. In this paper, our result depends on a stronger assumption $\alpha+2H>2$. The assumption arises since we apply the Cauchy-Schwarz inequality in \eqref{eq:on_good_set1} and \eqref{eq:on_good_set2}, where we decouple two independent renewal processes. The same issue arises in \cite{LSWZ26,li.song.ea:26:on}, and it was shown in \cite[Lemma 4.20]{li.song.ea:26:on} that the estimate obtained after applying the Cauchy--Schwarz inequality is optimal for establishing the finiteness of  the so-called \textit{self-energy}. Since the finiteness of the self-energy plays a critical role in proving uniform integrability therein, extending the disorder-relevant regime from $\alpha+2H>2$ to $\alpha+H>1$ requires a new approach to estimating the partition function that does not rely on the finiteness of the self-energy.

Note that in \cite{LSWZ26}, it was shown that for the chaos expansion of the Wick-ordered partition function, each order of chaos is well-defined when $\alpha+H>1$, although the chaos expansion converges in $L^2$ only for $\alpha\ge \frac12$. In addition, a renormalization argument in \cite{li.song.ea:26:on} also suggests $\alpha+H=1$ is the threshold between the disorder-relevant regime and disorder-irrelevant regime. Hence, we believe that  on the regime (the orange regime in the phase diagram \ref{Fig:phase-plane})
\begin{equation*}
\Big\{(\alpha,H):\alpha+H>1,\alpha+2H\leq2, \alpha\in\Big(0,\frac12\Big], H\in\Big(\frac12,1\Big)\Big\},
\end{equation*}
disorder is relevant, but it is still not clear whether it is an $L^1$-regime as $\alpha+2H>2$.

\textbf{(Disorder-irrelevant regime)} It is conjectured by the Weinrib--Halperin prediction that the gray regime in the phase diagram~\ref{Fig:phase-plane}
\begin{equation*}
\Big\{(\alpha,H): \alpha\in\Big(0,\frac12\Big], \alpha+H<1\Big\}
\end{equation*}
is disorder-irrelevant. The study of this regime remains largely open. A classical criterion to confirm disorder irrelevance is to show that for sufficiently small $\beta>0$, the critical exponent of the quenched free energy is the same as that of the free energy with $\beta=0$ (see \cite{Lac10}). However, in \cite{Ber13}, it was shown that when $H>\frac12$, there is no critical exponent for the quenched free energy. Hence, some new criterion for the disorder-irrelevant regime is needed when $H>\frac12$. When $H<\frac12$, \cite[Theorem~5]{Ber13} showed that, at least, there is no contradiction to disorder irrelevance. Nevertheless, we expect that the results in \cite{Lac10} also hold when $H<\frac12$.

\subsection{Organization}\label{sec:organ}
The rest of the paper is organized as follows. In Section~\ref{sec:pre}, we present some preliminary results, with particular emphasis on the fractional SHE \eqref{def:she}, and prove Theorem \ref{thm:main} assuming two crucial intermediate results, namely, Propositions \ref{prop:finite_conv} and \ref{prop:uni_error}. We then prove Propositions \ref{prop:finite_conv} and \ref{prop:uni_error} in Sections \ref{sec:proof_prop1} and \ref{sec:proof_prop2}, respectively.

\section{Preliminaries and  proof of Theorem~\ref{thm:main}}\label{sec:pre}
In this section, we first recall how to construct the $L^1$-solution to the fractional SHE \eqref{def:she} in the mild sense \eqref{def:mild}. Although the development of the whole $L^1$-theory in \cite{QRV25} is not an easy task, the expression of the $L^1$-solution that we need is much more elementary. We then prove Theorem~\ref{thm:main}. The proof relies on two crucial intermediate results, namely, Propositions \ref{prop:finite_conv} and \ref{prop:uni_error}, whose proofs are deferred to subsequent sections.

\subsection{Preliminaries}
In this section, we collect the notations and known facts that will be used throughout the paper. In particular, we provide an explicit construction of the $L^1$-solution of the fractional SHE \eqref{def:she}.

Suppose that the fractional SHE \eqref{def:she} lives on  $[0,T]\times\bbR$ for some fixed $T>0$. Let $\cH$ be the Hilbert space associated with the Gaussian noise $\xi$, which is the completion of smooth functions with compact support under the inner product
\begin{equation}\label{def:H_inner_prod}
\langle f,g\rangle_\cH:=\int_0^T\int_0^Tf(s)g(t)|t-s|^{2H-2}\dd s\dd t
\end{equation}
for some $H\in(\frac12,1)$, and we denote the norm on $\cH$ by $\|\cdot\|_\cH$. Suppose $\{e_k(t),t\in[0,T]\}_{k\geq1}$ is an orthonormal basis of $\cH$ consisting of bounded continuous functions. Then, for any $f\in\cH$, we have the orthogonal expansion 
$f=\sum_{k=1}^\infty\langle f,e_k\rangle_\cH e_k$.  

In \cite{li.song.ea:26:on}, the $L^1$-solution of \eqref{def:mild} was constructed as follows.
\begin{enumerate}
\item The equation \eqref{def:mild} was first projected onto the $n$-dimensional truncation of the noise $\xi$. To be precise, define
\begin{equation}\label{def:project_noise}
\xi_k:=\xi(e_k)=\int_0^Te_k(t)\xi(t)\dd t.
\end{equation}
It was shown that if $u$ solves \eqref{def:mild}, then $u_n:=\bbE[u|\boldsymbol{\xi}^{(n)}]$ solves the corresponding projected equation, where $\boldsymbol{\xi}^{(n)}:=(\xi_1,\cdots,\xi_n)$.
\item There is an explicit expression for $u_n$. Denote
\begin{equation}\label{eq:tilde_e}
\tilde{e}_k(t):=\int_0^Te_k(s)|t-s|^{2H-2}\dd s.
\end{equation}
Let $L_t$ be the local time process of $X_t$ at $x=0$, and define
\begin{equation}\label{def:m}
m_k(t)=m_k^{(\beta)}(t):=\beta\int_0^t\tilde{e}_k(s)\dd L_s.
\end{equation}
Then
\begin{equation}\label{def:u_n}
u_n(t,x):=\bE_{(t,x)}\Big[\exp\Big\{\sum\limits_{k=1}^n\Big(m_k(t)\xi_k-\frac12 m_k(t)^2\Big)\Big\}\Big],
\end{equation}
which is well-defined for $\rho\in(\frac{1}{H},2]\Longleftrightarrow\alpha\in(1-H,\frac12]$ (see \cite[Section 4]{li.song.ea:26:on}).

\item It is straightforward to see that $u_n(t,x)$ is a non-negative martingale. The final step is to show that $u_n$ (for fixed $(t,x)$) is uniformly integrable, which is the case when $\rho\in(\frac{1}{2H-1},2]\Longleftrightarrow\alpha+2H>2$ and $\alpha\in(0,\frac12]$. Then the $L^1$-solution admits a Feynman-Kac formula
\begin{equation}\label{eq:feynman_kac_u}
u(t,x):\xlongequal[\bbP\text{-a.s.}]{L^1}\lim\limits_{n\to\infty}u_n=\bE_{(t,x)}\Big[\exp\Big\{\sum\limits_{k=1}^\infty\Big(m_k(t)\xi_k-\frac12 m_k(t)^2\Big)\Big\}\Big].
\end{equation}
\end{enumerate}

For the full details of the $L^1$-Skorohod integral, we refer to \cite{QRV25,li.song.ea:26:on}.

\subsection{Proof of Theorem~\ref{thm:main}}\label{sec:strategy}
The proof of Theorem~\ref{thm:main} consists of the following three steps.

{\bf Step 1.} For fixed $n\geq1$, we approximate $\sigma(\xi_1,\dots,\xi_n)$ by a linear combination of $\omega_1,\cdots,{\omega_{\lfloor tN\rfloor}}$ through $e_1,\cdots,e_n$, denoted by {$\xi_{1,N},\cdots,\xi_{n,N}$}, as $N\to\infty$. Then by projecting $\what{Z}_{\lfloor Nt\rfloor}^{\beta_N}$ onto {$\sigma(\xi_{1,N},\cdots,\xi_{n,N})$}, we show that the projected partition function converges in distribution to $u_n(t,0)$. More precisely, we set 
\begin{align}
\label{def:xi_kN}\xi_{k,N}&:=N^{-H}\sum\limits_{j=1}^{{\lfloor Nt\rfloor}} I_{j,N}(e_k)\omega_j,\quad\text{for}~1\leq k\leq n,
\end{align}
where
\begin{align}
\label{def:I_jN}I_{j,N}(f)&:=N\int_{\frac{j-1}{N}}^\frac{j}{N}f(t-s)\dd s,\quad\text{for}~1\leq j\leq Nt,
\end{align}
where we apply a time reversal for $I_{j,N}$ since $u_n(t,0)$ is conditioned at the terminal.
Then we define
\begin{equation}\label{def:project_Z}
\what{Z}_{N,n}(t):=\bbE\Big[\what{Z}_{\lfloor Nt\rfloor}^{\beta_N}\Big|\xi_{1,N},\cdots,\xi_{n,N}\Big]. 
\end{equation}

\begin{remark}\label{extension-f}
Note that not every element of $\cH$ can be identified with a measurable function; in fact, $\cH$ may contain distributions (see~\cite{pt00}). We remark that the integral in \eqref{def:I_jN} can be extended to all $f\in\cH$ through continuous extension. More precisely, we first define $I_{j,N}$ on step functions $f$ by~\eqref{def:I_jN}. For such $f$, we can prove there exists a positive constant $C_{H,N}$ such that 
\[
|I_{j,N}(f)|
\leq C_{H,N}\|f\|_{\cH}.
\]
Indeed,  by a change of variables, Parseval's identity, the
Cauchy--Schwarz inequality, and the Fourier representation of the
$\cH$-norm,
\begin{align*}
|I_{j,N}(f)|^2
&=
C N^2
\left|
\int_{\bbR}
\widehat f(\xi)
\overline{
{\widehat\ind_{[t-j/N,\,t-(j-1)/N]}}(\xi)
}
\dd\xi
\right|^2\\
&=
C N^2
\left|
\int_{\bbR}
\widehat f(\xi)|\xi|^{\frac{1-2H}{2}}
\overline{
{\widehat\ind_{[t-j/N,\,t-(j-1)/N]}}(\xi)
}
|\xi|^{\frac{2H-1}{2}}
\dd\xi
\right|^2\\
&\leq
C N^2
\left(
\int_{\bbR}
|\widehat f(\xi)|^2|\xi|^{1-2H}\dd\xi
\right)
\left(
\int_{\bbR}
\left|
{\widehat\ind_{[t-j/N,\,t-(j-1)/N]}}(\xi)
\right|^2
|\xi|^{2H-1}\dd\xi
\right)\\
&\leq
C_H N^2\|f\|_{\cH}^2
\int_{\bbR}
\left|
{\widehat\ind_{[t-j/N,\,t-(j-1)/N]}}(\xi)
\right|^2
|\xi|^{2H-1}\dd\xi.
\end{align*}
Since the interval $[t-j/N,t-(j-1)/N]$ has length $1/N$, we have
\[
\left|
{\widehat\ind_{[t-j/N,\,t-(j-1)/N]}}(\xi)
\right|
\leq
C\min\left\{\frac1N,\frac1{|\xi|}\right\}.
\]
Therefore, since $2H-1>-1$ and $2H-3<-1$,
\begin{align*}
\int_{\bbR}
\left|
{\widehat\ind_{[t-j/N,\,t-(j-1)/N]}}(\xi)
\right|^2
|\xi|^{2H-1}\dd\xi
&\leq
C\int_{|\xi|\leq N}\frac{1}{N^2}|\xi|^{2H-1}\dd\xi
+
C\int_{|\xi|>N}|\xi|^{2H-3}\dd\xi\leq C_H N^{2H-2},
\end{align*}
 Hence,
\[
|I_{j,N}(f)|
\leq C_{H,N} \|f\|_{\cH}.
\]
Thus, $I_{j,N}$ is continuous with respect to the $\cH$-norm on the
space of step functions. Since step functions are dense in $\cH$,
$I_{j,N}$ admits a unique continuous extension to $\cH$, which we still
denote by $I_{j,N}$. Accordingly, for a general $f\in\cH$,
\eqref{def:I_jN} is understood in this extended sense.
\end{remark}

The first key ingredient in the proof is the  following proposition, of which the proof is deferred to Section~\ref{sec:proof_prop1}.
\begin{proposition}\label{prop:finite_conv}
Let {$\rho\in(\frac1H,2]\Longleftrightarrow\alpha\in(1-H,\frac12]$} and $H\in(\frac12,1)$, and $\beta_N$ be given by \eqref{eq:scale_beta}. Then for any fixed $t>0$ and $n\geq1$, as $N\to\infty$, $\what{Z}_{N,n}(t)\overset{\text{(d)}}{\longrightarrow}u_n^{({\kappa_\alpha}\hat{\beta})}(t,0)$.
\end{proposition}

{\bf Step 2.} As $u_n(t,x)$ converges to $u(t,x)$ in $L^1$, it remains to show that $\what{Z}_{\lfloor Nt\rfloor}^{\beta_N}$ can be approximated by $\what{Z}_{N,n}(t)$ uniformly in $N$. Then we have the following proposition (see Section~\ref{sec:proof_prop2} for its proof).
\begin{proposition}\label{prop:uni_error}
Let $\alpha\in(0,\frac12], H\in(\frac12,1)$ and $\alpha+2H>2$, and $\beta_N$ be given by \eqref{eq:scale_beta}. We have that
\begin{equation*}
\lim\limits_{n\to\infty}\sup\limits_{N\geq1}\bbE\Big|\what{Z}_{\lfloor Nt\rfloor}^{\beta_N}-\what{Z}_{N,n}(t)\Big|=0.
\end{equation*}
\end{proposition}

{\bf Step 3.} Now we are ready to prove Theorem \ref{thm:main}.  Without loss of generality, we work on the case $t=1$ throughout the rest of the paper, and simply denote $\what{Z}_{N,n}(t)$ by $\what{Z}_{N,n}$, and $u_n^{(\kappa_\alpha\hat{\beta})}(t,0), u^{(\kappa_\alpha\hat{\beta})}(t,0)$ by $u_n, u$, respectively.

For any bounded  $1$-Lipschitz function $f:\bbR\to\bbR$, we have 
\begin{equation}\label{eq:prove_thm}
\begin{split}
&\bigg|\bbE\Big[f\Big(\what{Z}_{N}^{\beta_N}\Big)\Big]-\bbE[f(u)]\bigg|\leq\bbE\bigg|\what{Z}_{N}^{\beta_N}-\what{Z}_{N,n}\bigg|+\bigg|\bbE\Big[f\Big(\what{Z}_{N,n}\Big)\Big]-\bbE[f(u_n)]\bigg|+\bbE\big|u_n-u\big|.
\end{split}
\end{equation}
For any $\gep>0$, we can choose  $n=n_0$ sufficiently large, such that the first and the last terms on the right-hand side of \eqref{eq:prove_thm} are smaller than $\gep/3$ by Proposition \ref{prop:uni_error} and eq.~\eqref{eq:feynman_kac_u} respectively. For the second term, we can choose $N_0:=N(n_0)$, such that for $N\geq N_0$, it is also smaller than $\gep/3$ by Proposition~\ref{prop:finite_conv}. Therefore, \eqref{eq:prove_thm} is smaller than $\gep$ for all $N$  large enough, which completes the proof.


\section{Proof of Proposition \ref{prop:finite_conv}}\label{sec:proof_prop1}
In this section, we prove Proposition \ref{prop:finite_conv}. Throughout this section, the Gaussian computations are carried out under $\bbP$ for each fixed realization of the renewal process $\tau$. For simplicity, we suppress the dependence  on $\tau$ whenever no confusion can arise.

First, by a straightforward computation, we derive a more explicit expression for $\what{Z}_{N,n}$ defined in~\eqref{def:project_Z}. For each fixed $n$, denote $\boldsymbol{\xi}_N=\boldsymbol{\xi}^{(n)}_N:=(\xi_{1,N},\cdots,\xi_{n,N})^{T}$, recalling that $\xi_{k,N}$ is defined in~\eqref{def:xi_kN}. We can decompose the Hamiltonian $$Y_N:=\beta_N\sum_{k=1}^N\omega_k\ind_{\{k\in\tau\}}$$ as
\begin{equation}\label{eq:def_XN}
Y_N=a_N^{T}\boldsymbol{\xi}_N+R_N,
\end{equation}
where $a_N=a_N^{(n)}$ is a vector to be determined and $R_N=R_N^{(n)}$ is a  Gaussian random variable  independent of  $\boldsymbol{\xi}_N$. Then
\begin{equation}\label{eq:gaussian_decomp}
0=\cov_\bbP(\boldsymbol{\xi}_N,R_N)=\cov_\bbP(\boldsymbol{\xi}_N,Y_N)-\cov_\bbP(\boldsymbol{\xi}_N)a_N.
\end{equation}
Note that
\begin{equation}\label{eq:cov_xi}
\begin{split}
\cov_{\bbP}(\xi_{i,N},\xi_{j,N})&=N^{2-2H}\sum\limits_{m=1}^N\sum\limits_{\ell=1}^N\int_{\frac{m-1}{N}}^{\frac{m}{N}}\int_{\frac{\ell-1}{N}}^{\frac{\ell}{N}}e_i(1-r)e_j(1-s)\gamma(m-\ell)\dd s\dd r\\
&\xrightarrow{N\to\infty}\langle e_i,e_j\rangle_\cH=\begin{cases}
1,\quad\text{if}~i=j,\\[2pt]
0,\quad\text{if}~i\neq j,
\end{cases}
\end{split}
\end{equation}
by a Riemann sum approximation. Hence, for fixed $n$, $C_N=C_{N,n}:=\cov_\bbP(\boldsymbol{\xi}_N)$ is invertible for large enough $N$. Therefore, by \eqref{eq:gaussian_decomp},
\begin{equation}\label{eq:a_b}
a_N= C_N^{-1}b_N,\quad\text{with}~b_N:=\cov_\bbP(\boldsymbol{\xi}_N,Y_N).
\end{equation}

Then, noting that $R_N$ is independent of $\boldsymbol{\xi}_N$, we have  $Y_N|_{\boldsymbol{\xi}_N=\boldsymbol{z}}\overset{\mathrm{(d)}}{=}a_N^T\boldsymbol{z}+R_N$. Furthermore, since $\bbE[R_N]=\bbE[Y_N]-a_N^T\bbE[\boldsymbol{\xi}_N]=0$ we get
\begin{equation}
\begin{split}
\var_\bbP(R_N)&=\var_\bbP(Y_N)-2a_N^Tb_N+a_N^TC_Na_N\\
&=\var_\bbP(Y_N)-b_N^TC_N^{-1}b_N,
\end{split}
\end{equation}
where we have used the symmetry of $C_N$ in the last equality. Therefore,  the conditional distribution of $Y_N$ given $\boldsymbol{\xi}_N$ is 
\begin{equation}
Y_N|\boldsymbol{\xi}_N\sim\cN(a_N^T\boldsymbol{\xi}_N, \var_\bbP(Y_N)-b_N^TC_N^{-1}b_N). 
\end{equation}
Thus, we get
\begin{equation}\label{eq:cond_Z}
\begin{split}
\what{Z}_{N,n}&=\bbE\Big[\bE\Big[\exp\Big\{Y_N-\frac12\var_\bbP(Y_N)\Big\}\Big]\Big|\boldsymbol{\xi}_N\Big]=\bE\Big[\exp\Big\{-\frac12\var_\bbP(Y_N)\Big\}\bbE\Big[\exp(Y_N)\Big|\boldsymbol{\xi}_N\Big]\Big]\\
&=\bE\Big[\exp\Big\{-\frac12\var_\bbP(Y_N)\Big\}\exp\Big\{a_N^T\boldsymbol{\xi}_N+\frac12\Big(\var_\bbP(Y_N)-b_N^TC_N^{-1}b_N\Big)\Big\}\Big]\\
&=\bE\Big[\exp\Big\{a_N^T\boldsymbol{\xi}_N-\frac12 b_N^T C_N^{-1}b_N\Big\}\Big]=\bE\Big[\exp\Big\{(\sL_Nb_N)^T \sL_N\boldsymbol{\xi}_N-\frac12 (\sL_N b_N)^T \sL_Nb_N\Big\}\Big],
\end{split}
\end{equation}
where $C_N^{-1}=\sL_N^T\sL_N$ is the Cholesky decomposition, so that $\sL_N\boldsymbol{\xi}_N$ is an random vector with i.i.d.\ standard Gaussian components.

Recall from \eqref{def:u_n} that
\begin{equation*}
u_n=u_n^{(\kappa_\alpha\hat{\beta})}(1,0)=\bE_{(1,0)}\Big[\exp\Big\{\sum\limits_{k=1}^n\Big(m_k^{(\kappa_\alpha\hat{\beta})}(1)\xi_k-\frac12 m_k^{(\kappa_\alpha\hat{\beta})}(1)^2\Big)\Big\}\Big].
\end{equation*}

To show that
\begin{equation}\label{eq:Z_to_u}
\what{Z}_{N,n}\overset{\text{(d)}}{\longrightarrow}u_n,\quad\text{as}~ N\to\infty, 
\end{equation}
we claim that it suffices to show that
\begin{equation}\label{eq:Lb_to_m}
\sL_Nb_N\overset{\text{(d)}}{\longrightarrow} \boldsymbol m=(m_1,\cdots,m_n)^T,
\end{equation}
where $m_k:=m_k^{( \kappa_\alpha\hat{\beta})}(1)$ for $1\leq k\leq n$ is defined in \eqref{def:m}.

We postpone the proof for \eqref{eq:Lb_to_m} and apply it to prove \eqref{eq:Z_to_u}. Denote
\begin{align}
F_N(\boldsymbol{x})&:=\bE\Big[\exp\Big\{(\sL_Nb_N)^T\boldsymbol{x}-\frac12 (\sL_Nb_N)^T \sL_Nb_N\Big\}\Big],\\
F(\boldsymbol{x})&:=\bE_{(1,0)}\Big[\exp\Big\{\sum\limits_{k=1}^n\Big(m_kx_k-\frac12 m_k^2\Big)\Big\}\Big],
\end{align}
where $\boldsymbol{x}:=(x_1,\cdots,x_n)^T$.

For any fixed $n$ and $\boldsymbol{x}$, $\exp(\boldsymbol{z}^T\boldsymbol{x}-\frac12\boldsymbol{z}^T\boldsymbol{z})$ is a bounded continuous function in $\boldsymbol{z}$. Hence, by~\eqref{eq:Lb_to_m}, we have that $\lim_{N\to\infty}F_N(\boldsymbol{x})=F(\boldsymbol{x})$. Moreover, noting that both $\sL_N\boldsymbol{\xi}_N$ and $\boldsymbol{\xi}^{(n)}$ have the same Gaussian measure $\mu_n(\dd\boldsymbol{x}):=(2\pi)^{- n/2}\exp(-\frac12\boldsymbol{x}^T\boldsymbol{x})$, we get 
\begin{equation}
\int_{\bbR^n}F_N(\boldsymbol{x})\mu_n(\dd\boldsymbol{x})=\bbE[F_N(\sL_N\boldsymbol{\xi_N})]=\bbE[F(\boldsymbol{\xi}^{(n)})]=1.
\end{equation}
Then Scheff\'{e}'s lemma implies that
\begin{equation}\label{eq:scheffe}
\lim\limits_{N\to\infty}\int_{\bbR^n}\Big|F_N(\boldsymbol{x})-F(\boldsymbol{x})\Big|\mu_n(\dd\boldsymbol{x})=0.
\end{equation}
Finally, for any bounded and $1$-Lipschitz function $f:\bbR\to\bbR$, we have that
\begin{equation}
\begin{split}
&\Big|\bbE\big[f\big(F_N(\sL_N\boldsymbol{\xi}_N)\big)\big]-\bbE\big[f\big(F(\boldsymbol{\xi}^{(n)})\big)\big]\Big|\\
=&\bigg|\int_{\bbR^n}f\big(F_N(\boldsymbol{x})\big)\mu_n(\dd\boldsymbol{x})-\int_{\bbR^n}f\big(F(\boldsymbol{x})\big)\mu_n(\dd\boldsymbol{x})\bigg|\\
\leq&\int_{\bbR^n}\Big|f\big(F_N(\boldsymbol{x})\big)-f\big(F(\boldsymbol{x})\big)\Big|\mu_n(\dd\boldsymbol{x})\leq \int_{\bbR^n}\Big|F_N(\boldsymbol{x})-F(\boldsymbol{x})\Big|\mu_n(\dd\boldsymbol{x}),
\end{split}
\end{equation}
which tends to $0$ as $N\to\infty$ by \eqref{eq:scheffe}. This concludes \eqref{eq:Z_to_u}.

It only remains to prove \eqref{eq:Lb_to_m}. We first provide an explicit formula for $b_N=\cov_\bbP(Y_N,\boldsymbol{\xi}_N)$. We have that for any $1\leq j\leq n$,
\begin{equation}\label{eq:formula_b}
\begin{split}
b_{j,N}=\cov_\bbP(Y_N,\xi_{j,N})=\hat{\beta}\frac{L(N)}{N^{\alpha+2H-2}}\sum\limits_{\ell=1}^N\sum\limits_{i=1}^N\int_{\frac{i-1}{N}}^{\frac{i}{N}}\ind_{\{\ell\in\tau\}}\gamma(i-\ell)e_j(1-s)\dd s.
\end{split}
\end{equation}
By Cram\'{e}r-Wold theorem, to prove \eqref{eq:Lb_to_m}, it suffices to show that
\begin{equation}\label{eq:cramer_wold}
\boldsymbol{z}^T\sL_Nb_N\overset{(d)}{\longrightarrow}\sum\limits_{k=1}^n z_km_k, \quad\text{as}~N\to\infty,~\text{for any}~\boldsymbol{z}=(z_1,\cdots,z_n)\in\bbR^n. 
\end{equation}

We apply Hamburger's moment method to show \eqref{eq:cramer_wold} below. Note that for any integer $m\in\bbN$,
\begin{equation}\label{eq:mix_mmt}
\begin{split}
\lim\limits_{N\to\infty}\bE\Big[\big(\boldsymbol{z}^T\sL_N b_N\big)^m\Big]=\lim\limits_{N\to\infty}\sum\limits_{k_1=1}^n\cdots\sum\limits_{k_m=1}^n\Big(\prod\limits_{j=1}^m(\boldsymbol{z}^T\sL_N)_{k_j}\Big)\bE\Big[\prod\limits_{j=1}^m b_{k_j,N}\Big].
\end{split}
\end{equation}
Since $\lim_{N\to\infty}\sL_N=I_n$, we have $\lim_{N\to\infty}(\boldsymbol{z}^T \sL_N)=\boldsymbol{z}^T$. Thus, to obtain the convergence of  moments, it suffices to show
\begin{equation}\label{eq:conv_mmt}
\lim\limits_{N\to\infty}\bE\Big[\prod\limits_{j=1}^m b_{k_j,N}\Big]=\bE\Big[\prod\limits_{j=1}^m m_{k_j}\Big],
\end{equation}
for any $m\geq1$ and $1\leq k_1,\cdots,k_m\leq n$.

By \eqref{eq:formula_b}, we get
\begin{equation}\label{eq:mix_mmt2}
\begin{split}
\bE\Big[\prod\limits_{j=1}^m b_{k_j,N}\Big]=\hat{\beta}^m&\sum\limits_{\ell_1=1}^N\sum\limits_{i_1=1}^N\cdots\sum\limits_{\ell_m=1}^N\sum\limits_{i_m=1}^N\frac{L(N)^m}{N^{(\alpha-1)m}}\bP(\ell_1\in\tau,\cdots,\ell_m\in\tau)\frac{1}{N^m}\\
&\times\int_{\frac{i_1-1}{N}}^{\frac{i_1}{N}}\cdots\int_{\frac{i_m-1}{N}}^{\frac{i_m}{N}}\prod\limits_{j=1}^m\frac{1}{N^{2H-2}}\gamma(i_j-\ell_j)e_{k_j}(1-s_j)\dd s_j.
\end{split}
\end{equation}
Then by \eqref{eq:local_limit}, a Riemann sum approximation (see \cite[Page 2808--2809]{LSWZ26} for more details), as $N\to\infty$, \eqref{eq:mix_mmt2} converges to
\begin{equation}\label{eq:mix_mmt3}
\hat{\beta}^m\int_{[0,1]^m}\int_{[0,1]^m}\psi_m(t_1,\cdots,t_m)\prod\limits_{j=1}^m|t_j-s_j|^{2H-2}e_{k_j}(1-s_j)\dd s_j\dd t_j,
\end{equation}
where $\psi_m(t_1,\cdots,t_m)$ is $0$ if $t_i=t_j$ for some $i\neq j$, and is symmetric, in the sense that for any $(t_1,\cdots,t_m)$ with $t_i\neq t_j$ for  $i\neq j$, there exists a unique permutation $\sigma$ on $\{1,\cdots,m\}$, such that $0<t_{\sigma(1)}<\cdots<t_{\sigma(m)}<1$, and
\begin{equation}\label{def:psi}
\psi_m(t_1,\cdots,t_m)=\frac{C_\alpha^m}{t_{\sigma(1)}^{1-\alpha}\prod_{k=2}^m(t_{\sigma(k)}-t_{\sigma(k-1)})^{1-\alpha}}.
\end{equation}
Note that $\psi_m(t_1,\cdots,t_m)$ arises naturally due to the local limit theorem \eqref{eq:local_limit}. For any $(t_1,\cdots,t_m)\in[0,1]^m$ with $t_i\neq t_j$ for any $i\neq j$,
\begin{equation}
\psi_m(t_1,\cdots,t_m)=\lim\limits_{N\to\infty} \frac{L(N)^m}{N^{(\alpha-1)m}}\bP\big([t_1 N]\in\tau,\cdots,[t_mN]\in\tau\big).
\end{equation}



It is straightforward to check that by a change of variables $s_j'=1-s_j$ and $t_j'=1-t_j$ for all $1\leq j\leq m$ in \eqref{eq:mix_mmt3}, the $2m$-fold integral is equal to $\bbE[\prod_{j=1}^m m_{k_j}]$. To see this, denote by $S_m$ the set of  all permutations on $\{1,\cdots,m\}$, and then by \cite[Lemma~4.2]{li.song.ea:26:on}  with $t_{\sigma(m+1)}:=1$, we have 
\begin{equation}\label{eq:local_time_mixed_moment}
\begin{split}
&\bE_{(1,0)}\Big[\prod_{j=1}^m m_{k_j}\Big]=(\kappa_\alpha\hat\beta)^m
\bE_{(1,0)}\bigg[\prod_{j=1}^m\bigg(\int_0^1\tilde e_{k_j}(t_j)\dd L_{t_j}\bigg)\bigg]\\
=&(\kappa_\alpha\hat\beta)^m\sum_{\sigma\in S_m}\int_{0<t_{\sigma(1)}<\cdots<t_{\sigma(m)}<1}
\prod_{j=1}^m\tilde e_{k_j}(t_j)g_\rho(t_{\sigma(j+1)}-t_{\sigma(j)},0)\dd t_j\\
=&(\kappa_\alpha\hat\beta)^m\sum_{\sigma\in S_m}\int_{[0,1]^m}\int_{0<t_{\sigma(1)}<\cdots<t_{\sigma(m)}<1}\prod_{j=1}^m e_{k_j}(s_j)|t_j-s_j|^{2H-2}
g_\rho(t_{\sigma(j+1)}-t_{\sigma(j)},0)
\dd s_j\dd t_j.
\end{split}
\end{equation}

Recall that $\alpha=1-\frac{1}{\rho}$ and $\kappa_\alpha g_\rho(1,0)=C_\alpha$ in Theorem \ref{thm:main}. By the scaling property of stable density function, we have $g_\rho(u,0)=g_\rho(1,0)u^{-1/\rho}=g_\rho(1,0)u^{\alpha-1}$. Then by $s'_j=1-s_j$ and $t'_j=1-t_j$,
\begin{align*}
\bE_{(1,0)}\Big[\prod_{j=1}^m m_{k_j}\Big]
&=\hat\beta^m
\int_{[0,1]^m}\int_{[0,1]^m}
\psi_m(t'_1,\cdots,t'_m)
\prod_{j=1}^m
|t_j'-s_j'|^{2H-2}
e_{k_j}(1-s_j')
\dd s'_j\dd t'_j,
\end{align*}
which verifies \eqref{eq:conv_mmt}.

Finally, to apply Hamburger's moment method, we need to verify the condition (see \cite[Theorem~3.3.12]{Dur19}) 
\begin{equation}\label{e:H-con}
\sup_{k\geq1} M_{2k}^{1/2k}/2k<\infty,
\end{equation}
where $M_k$ is the $k$-th moment of $\sum_{j=1}^n z_jm_j$.   For the $2k$-th moment, we have
\begin{align*}
M_{2k}
&=
\bE\left[
\left(\sum_{j=1}^n z_jm_j\right)^{2k}
\right] =
\sum_{k_1=1}^n\cdots\sum_{k_{2k}=1}^n
\left(\prod_{j=1}^{2k}z_{k_j}\right)
\bE\left[\prod_{j=1}^{2k}m_{k_j}\right].
\end{align*}
Since $n$ is fixed and $e_1,\ldots,e_n$ are bounded continuous
functions, we may bound all $|z_{k_j}|$ and
$\sup_{t\in[0,1]}|e_{k_j}(t)|$, for
$1\leq k_1,\ldots,k_{j}\leq n$, by a constant $C_{n,\boldsymbol{z}}$.
Using \eqref{eq:local_time_mixed_moment} above, we
first integrate with respect to $s_1,\ldots,s_{2k}$. For each
$j=1,\ldots,2k$ and fixed $t_j\in[0,1]$, we have
\begin{align*}
\int_0^1|t_j-s_j|^{2H-2}\dd s_j =
\frac{1}{2H-1}
\left(
t_j^{2H-1}+(1-t_j)^{2H-1}
\right)
\leq
\frac{2}{2H-1}.
\end{align*}
 Hence, after integrating out
$s_1,\ldots,s_{2k}$, all these contributions can be absorbed into a
constant $C_{n,z,\hat\beta,\alpha,H}$ depending only on $n,z,\hat\beta,\alpha$ and $H$. By symmetry, it suffices to consider the region
$0<t_1<\cdots<t_{2k}<1$, which gives a factor $(2k)!$. Therefore,
\begin{align*}
M_{2k}
\leq
C_{n,z,\hat\beta,\alpha,H}^{2k}(2k)!
\int_{0<t_1<\cdots<t_{2k}<1}
\frac{\dd t_1\cdots\dd t_{2k}}
{t_1^{1-\alpha}
\prod_{j=2}^{2k}(t_j-t_{j-1})^{1-\alpha}}.
\end{align*}
By \cite[Lemma A.5]{LSWZ26},
\begin{align*}
\int_{0<t_1<\cdots<t_{2k}<1}
\frac{\dd t_1\cdots\dd t_{2k}}
{t_1^{1-\alpha}
\prod_{j=2}^{2k}(t_j-t_{j-1})^{1-\alpha}}
=
\frac{\Gamma(\alpha)^{2k}}
{\Gamma(2k\alpha+1)}.
\end{align*}
Hence,
\begin{align*}
\frac{M_{2k}^{1/(2k)}}{2k}
\leq
\frac{C_{n,z,\hat\beta,\alpha,H}}{2k}
\left[
\frac{(2k)!\Gamma(\alpha)^{2k}}
{\Gamma(2k\alpha+1)}
\right]^{1/(2k)} \leq
\frac{C}{2k}
\left[
\frac{\sqrt{2k}(2k)^{2k}}
{\sqrt{2k\alpha}(2k\alpha)^{2k\alpha}}
\right]^{1/(2k)}\leq
\frac{C'}{(2k)^\alpha},
\end{align*}
where we have used Stirling's formula in the second inequality.
Since $\alpha>0$, the right-hand side is uniformly bounded in $k$, and this verifies condition~\eqref{e:H-con}. Thus the distribution of $\sum_{j=1}^n z_jm_j$ is determined by its
moments, and the moment method yields~\eqref{eq:cramer_wold}. This
proves \eqref{eq:Lb_to_m} and concludes the proof of Proposition
\ref{prop:finite_conv}.


\section{Proof of Proposition \ref{prop:uni_error}}\label{sec:proof_prop2}
In this section, we prove
\begin{equation}\label{eq:project_error1}
\lim\limits_{n\to\infty}\sup\limits_{N\geq1}\bbE\Big|\what{Z}_N^{\beta_N}-\what{Z}_{N,n}\Big|=0
\end{equation}
under the conditions assumed in Proposition \ref{prop:uni_error}.

Note that in \cite{LSWZ26}, to prove that  $\{\what{Z}_N^{\beta_N}\}_N$ is uniformly integrable, a ``good set'' of the renewal trajectories was introduced, which is
\begin{equation}\label{eq:good_set}
A_N^{(K)}:=\Big\{\tau:\beta_N^2\sum\limits_{n=1}^N\sum\limits_{m=1}^N\gamma(n-m)\ind_{\{n\in\tau\}}\ind_{\{m\in\tau\}}\leq K\Big\}.
\end{equation}
We define a partition function restricted on the good set $A_N^{(K)}$ by
\begin{equation}\label{def:truncated_Z}
\what{Z}_N^{(K)}:=\bE\Big[\exp\Big\{\Big[\sum\limits_{n=1}^N \beta_N\omega_n\ind_{\{n\in\tau\}}-\frac12\var_{\bbP}\Big(\sum\limits_{n=1}^N\beta_N\omega_n\ind_{\{n\in\tau\}}\Big)\Big]\ind_{A_N^{(K)}}\Big\}\Big].
\end{equation}
We emphasize that the indicator is on the exponent.

Then, to prove \eqref{eq:project_error1}, we use the following triangular inequality:
\begin{equation}\label{eq:project_error2}
\bbE\Big|\what{Z}_N^{\beta_N}-\what{Z}_{N,n}\Big|\leq\bbE\Big|\what{Z}_N^{\beta_N}-\what{Z}_N^{(K)}\Big|+\bbE\Big|\what{Z}_N^{(K)}-\what{Z}_{N,n}^{(K)}\Big|+\mathbb{E}\Big|\what{Z}_{N,n}-\what{Z}_{N,n}^{(K)}\Big|,
\end{equation}
where $\what{Z}_{N,n}^{(K)}:=\bbE[\what{Z}_N^{(K)}|\boldsymbol{\xi}_N]$ with $\boldsymbol{\xi}_N=\boldsymbol{\xi}_N^{(n)}:=(\xi_{1,N},\cdots,\xi_{n,N})$ (see \eqref{def:xi_kN}).

For the first term above, it has been shown by \cite[Equation (3.21)]{LSWZ26} that for some $C\in(0,\infty)$,
\begin{equation}\label{eq:term1}
\bbE\Big|\what{Z}_N^{\beta_N}-\what{Z}_N^{(K)}\Big|\leq C\sqrt{\bP\Big((A_N^{(K)})^c\Big)}+\bP\Big((A_N^{(K)})^c\Big).
\end{equation}
Moreover, by \cite[Equation (3.20)]{LSWZ26}, if $\alpha+2H>2$, then
\begin{equation}
\lim_{K\to\infty}\sup_{N\geq1}\bP\Big((A_N^{(K)})^c\Big)=0.
\end{equation}
Hence, for any $\gep>0$, we can choose $K$ large enough, such that
\begin{equation}\label{eq:term1a}
\sup\limits_{N\geq1}\bbE\Big|\what{Z}_N^{\beta_N}-\what{Z}_N^{(K)}\Big|<\frac{\gep}{3}.
\end{equation}

For the last term in \eqref{eq:project_error2}, we have that by Jensen's inequality
\begin{equation}
\mathbb{E}\Big|\what{Z}_{N,n}-\what{Z}_{N,n}^{(K)}\Big|=\bbE\Big|\bbE\Big[\what{Z}_N^{\beta_N}-\what{Z}_N^{(K)}\Big|\boldsymbol{\xi}_N\Big]\Big|\leq\bbE\Big|\what{Z}_N^{\beta_N}-\what{Z}_N^{(K)}\Big|,
\end{equation}
which has been controlled in \eqref{eq:term1a}.

Then Proposition \ref{prop:uni_error} is a direct consequence of the following lemma, which tackles the second term in \eqref{eq:project_error2}. 
\begin{lemma}
Let $\what{Z}_N^{(K)}$ and $\what{Z}_{N,n}^{(K)}$ be defined as above. Then for any fixed $K>0$,
\begin{equation}\label{eq:sup_L2_diff}
\lim\limits_{n\to\infty}\sup\limits_{N\geq1}\Big\|\what{Z}_N^{(K)}-\what{Z}_{N,n}^{(K)}\Big\|_{L^2(\bbP)}=0.
\end{equation}
\end{lemma}
\begin{proof}
First, note that the tower rule implies
\begin{equation}\label{eq:truncate_error1}
\begin{split}
\Big\|\what{Z}_N^{(K)}-\what{Z}_{N,n}^{(K)}\Big\|_{L^2(\bbP)}^2=&\bbE\Big[\Big(\what{Z}_N^{(K)}\Big)^2\Big]-2\bbE\Big[\what{Z}_N^{(K)}\bbE\Big[\what{Z}_N^{(K)}\Big|\boldsymbol{\xi}_N\Big]\Big]+\bbE\Big[\Big(\bbE\Big[\what{Z}_N^{(K)}\Big|\boldsymbol{\xi}_N\Big]\Big)^2\Big]\\
=&\bbE\Big[\Big(\what{Z}_N^{(K)}\Big)^2\Big]-2\bbE\Big[\bbE\Big\{\what{Z}_N^{(K)}\bbE\Big[\what{Z}_N^{(K)}\Big|\boldsymbol{\xi}_N\Big]\Big|\boldsymbol{\xi}_N\Big\}\Big]+\bbE\Big[\Big(\bbE\Big[\what{Z}_N^{(K)}\Big|\boldsymbol{\xi}_N\Big]\Big)^2\Big]\\
=&\bbE\Big[\Big(\what{Z}_N^{(K)}\Big)^2\Big]-\bbE\Big[\Big(\bbE\Big[\what{Z}_N^{(K)}\Big|\boldsymbol{\xi}_N\Big]\Big)^2\Big].
\end{split}
\end{equation}
A straightforward computation yields (also see \cite[Page 2812]{LSWZ26})
\begin{equation}
\bbE\Big[\Big(\what{Z}_N^{(K)}\Big)^2\Big]=\bE\Big[\exp\Big\{\ind_{A_N^{(K)}}\ind_{\big(A_N^{(K)}\big)'}\beta_N^2\sum\limits_{n=1}^N\sum\limits_{m=1}^N\gamma(n-m)\ind_{\{n\in\tau\}}\ind_{\{m\in\tau'\}}\Big\}\Big],
\end{equation}
where $\tau'$ is an independent copy of $\tau$ and $(A_N^{(K)})'$ is the corresponding good set \eqref{eq:good_set} for $\tau'$.

As in \eqref{eq:cond_Z}, we have 
\begin{equation}
\bbE\Big[\what{Z}_N^{(K)}\Big|\boldsymbol{\xi}_N\Big]=\bE\Big[\exp\Big\{\ind_{A_N^{(K)}}\Big((\sL_N b_N)^T\sL_N\boldsymbol{\xi}_N-\frac12(\sL_N b_N)^T \sL_N b_N\Big)\Big\}\Big],
\end{equation}
and thus
\begin{equation}
\bbE\Big[\Big(\bbE\Big[\what{Z}_N^{(K)}\Big|\boldsymbol{\xi}_N\Big]\Big)^2\Big]=\bE\Big[\exp\Big\{\ind_{A_N^{(K)}}\ind_{\big(A_N^{(K)}\big)'}(\sL_N b_N)^T(\sL_N b'_N)\Big\}\Big],
\end{equation}
where we recall that $b_N=\cov_{\bbP}(Y_N,\boldsymbol{\xi}_N)$ and $b'_N=\cov_{\bbP}(Y_N',\boldsymbol{\xi}_N)$, respectively, and $C_N^{-1}=\sL_N^T\sL_N$.

Then by the equality $e^{\ind_A X}=(e^X-1)\ind_A+1$, we have the following bounds for \eqref{eq:truncate_error1}
\begin{equation}\label{eq:truncate_error2}
\begin{split}
\Big\|\what{Z}_N^{(K)}-\what{Z}_{N,n}^{(K)}\Big\|_{L^2(\bbP)}^2=&\bE\Big[\exp\Big\{\ind_{A_N^{(K)}}\ind_{\big(A_N^{(K)}\big)'}\beta_N^2\sum\limits_{n=1}^N\sum\limits_{m=1}^N\gamma(n-m)\ind_{\{n\in\tau\}}\ind_{\{m\in\tau'\}}\Big\}\Big]\\
&-\bE\Big[\exp\Big\{\ind_{A_N^{(K)}}\ind_{\big(A_N^{(K)}\big)'}(\sL_N b_N)^T(\sL_N b'_N)\Big\}\Big]\\
=&\bE\Big[\ind_{A_N^{(K)}}\ind_{\big(A_N^{(K)}\big)'}\Big(\exp\Big\{\cov_\bbP(Y_N,Y'_N)\Big\}-\exp\Big\{(\sL_N b_N)^T(\sL_N b'_N)\Big\}\Big)\Big],
\end{split}
\end{equation}
where $Y_N$ was defined by \eqref{eq:def_XN}.

We first estimate the difference between the two exponents on the right-hand side of \eqref{eq:truncate_error2}. Recall from
\eqref{eq:def_XN} that
\begin{equation*}
Y_N=a_N^T\boldsymbol{\xi}_N+R_N=(\sL_Nb_N)^T\sL_N\boldsymbol{\xi}_N+R_N.
\end{equation*}
For the independent renewal trajectory $\tau'$, we set $R_N'=Y_N'-(\sL_Nb_N')^T\sL_N\boldsymbol{\xi}_N$ correspondingly. Since both $R_N$
and $R_N'$ are orthogonal to ${\xi}_N$ and $\sL_N$ is deterministic, we have that
\begin{equation}\label{eq:differ}
\cov_{\bbP}(Y_N,Y_N')-(\sL_Nb_N)^T(\sL_Nb_N')
=\cov_{\bbP}(R_N,R_N').
\end{equation}
Moreover, $\var_{\bbP}(R_N)=\var_{\bbP}(Y_N)-(\sL_Nb_N)^T(\sL_N b_N)$, and $\var_{\bbP}(R_N')=\var_{\bbP}(Y_N')-(\sL_Nb_N')^T(\sL_Nb_N')$. Therefore, by \eqref{eq:differ} and Cauchy--Schwarz inequality, we have that
\begin{equation}\label{eq:differ+}
\begin{split}
&\left|\cov_{\bbP}(Y_N,Y_N')-(\sL_Nb_N)^T(\sL_Nb_N')\right|\\
\leq&
\big[\var_{\bbP}(Y_N)-(\sL_Nb_N)^T(\sL_Nb_N)\big]^{\frac12}
\big[\var_{\bbP}(Y_N')-(\sL_Nb_N')^T(\sL_Nb_N')\big]^{\frac12}.
\end{split}
\end{equation}

On the other hand, by another application of Cauchy-Schwarz inequality, we have that
\begin{equation}\label{eq:on_good_set1}
\begin{split}
\big|\cov_{\bbP}(Y_N,Y_N')\big|^2\leq&\var_\bbP(Y_N)\var_\bbP(Y_N')\\
=&\Big[\beta_N^2\sum\limits_{n=1}^N\sum\limits_{m=1}^N\gamma(n-m)\ind_{\{n\in\tau\}}\ind_{\{m\in\tau\}}\Big]\Big[\beta_N^2\sum\limits_{n=1}^N\sum\limits_{m=1}^N\gamma(n-m)\ind_{\{n\in\tau'\}}\ind_{\{m\in\tau'\}}\Big],
\end{split}
\end{equation}
and
\begin{equation}\label{eq:on_good_set2}
\begin{split}
\big|(\sL_N b_N)^T(\sL_Nb_N')\big|^2\leq&(\sL_Nb_N)^T(\sL_Nb_N)(\sL_Nb_N')^T(\sL_Nb_N')\\
=&\big(\cov_\bbP(\boldsymbol{\xi}^T_N,Y_N)C_N^{-1}\cov_\bbP(\boldsymbol{\xi}_N,Y_N)\big)\big(\cov_\bbP(\boldsymbol{\xi}^T_N,Y_N')C_N^{-1}\cov_\bbP(\boldsymbol{\xi}_N,Y_N')\big)\\
\leq&\var_\bbP(Y_N)\var_\bbP(Y_N')\\
=&\Big[\beta_N^2\sum\limits_{n=1}^N\sum\limits_{m=1}^N\gamma(n-m)\ind_{\{n\in\tau\}}\ind_{\{m\in\tau\}}\Big]\Big[\beta_N^2\sum\limits_{n=1}^N\sum\limits_{m=1}^N\gamma(n-m)\ind_{\{n\in\tau'\}}\ind_{\{m\in\tau'\}}\Big],
\end{split}
\end{equation}
where the second inequality is due to the following lemma.
\begin{lemma}\label{lem:matrix_Cauchy_Schwarz}
Let $Y$ be an $n$-dimensional column random vector and $X$ be a random variable. If $\var(Y)$ is invertible, then $\cov(Y^T,X)[\var(Y)]^{-1}\cov(Y,X)\leq\var(X)$.
\end{lemma}
\begin{proof}[Proof of Lemma \ref{lem:matrix_Cauchy_Schwarz}]
Take a vector $a=[\var(Y)]^{-1}\cov(Y,X)$. 
Then Cauchy--Schwarz inequality yields
\begin{equation*}
a^T\cov(Y,X)=\cov(a^TY,X)\leq\sqrt{\big[a^T\var(Y)a\big]\var(X)},
\end{equation*}
which completes the proof.
\end{proof}

By \eqref{eq:on_good_set1} and \eqref{eq:on_good_set2}, on the event  $A_N^{(K)}\times (A_N^{(K)})'$, we have that, by the definition of the good set in~\eqref{eq:good_set},
\begin{equation*}
\big|\cov_{\bbP}(Y_N,Y_N')\big|\leq K,\quad\text{and}\quad
\big|(\sL_Nb_N)^T(\sL_Nb_N')\big|\leq K.
\end{equation*}
Hence, by \eqref{eq:differ+} and the Lagrange mean value theorem, 
\begin{equation}\label{eq:exponential-error}
\begin{split}
&\Big|\ind_{A_N^{(K)}}\ind_{\big(A_N^{(K)}\big)'}\Big(\exp\Big\{\cov_\bbP(Y_N,X'_N)\Big\}-\exp\Big\{(\sL_N b_N)^T(\sL_N b'_N)\Big\}\Big)\Big|\\
\leq&e^K
\big[\var_{\bbP}(Y_N)-(\sL_Nb_N)^T(\sL_Nb_N)\big]^{\frac12}
\big[\var_{\bbP}(Y_N')-(\sL_Nb_N')^T(\sL_Nb_N')\big]^{\frac12}.
\end{split}
\end{equation}

In view of \eqref{eq:truncate_error2}--\eqref{eq:exponential-error}, we have that by the independence between $\tau$ and $\tau'$ and Jensen's inequality,
\begin{equation*}
\begin{split}
\Big\|\what{Z}_N^{(K)}-\what{Z}_{N,n}^{(K)}
\Big\|_{L^2(\bbP)}^2
&\leq e^K\bE\Big[\Big(\var_{\bbP}(Y_N)-(\sL_Nb_N)^T(\sL_Nb_N)\Big)^{\frac12}\Big]^2\\
&\leq e^K\bE\Big[\var_{\bbP}(Y_N)-(\sL_Nb_N)^T(\sL_Nb_N)\Big].
\end{split}
\end{equation*}

Now we estimate the expectation on the right-hand side above. Note that
\begin{align}\label{eq:total_covariance_discrete}
\bE\big[\var_{\bbP}(Y_N)\big]
&=\beta_N^2\sum_{m=1}^N\sum_{n=1}^N
\gamma( m- n)\bP(m\in\tau,n\in\tau)
\end{align}
Then by \cite[Lemma 3.4]{LSWZ26}, \eqref{eq:local_limit}, a Riemann sum approximation and the dominated convergence theorem,
\begin{equation}
\lim\limits_{N\to\infty}\bE\big[\var_{\bbP}(Y_N)\big]=
2(C_\alpha\hat\beta)^2
\int_{0<r<s<1}
r^{\alpha-1}(s-r)^{\alpha+2H-3}\dd r\dd s.
\end{equation}
By \cite[Lemma 4.2, Corollary 4.3]{li.song.ea:26:on},
\begin{equation}
\bE_{(1,0)}\bigg[\int_0^1\int_0^1|r-s|^{2H-2}\dd L_r\dd L_s\bigg]=2g_\rho(1,0)^2\int_{0<r<s<1}r^{\alpha-1}(s-r)^{\alpha+2H-3}\dd r\dd s.
\end{equation}
Then, we can conclude that
\begin{equation}
\lim_{N\to\infty}\bE\big[\var_{\bbP}(Y_N)\big]
=
\bE_{(1,0)}\left[
({\kappa_\alpha}\hat\beta)^2
\int_0^1\int_0^1
|r-s|^{2H-2}\dd L_r\dd L_s
\right].
\end{equation}

On the other hand, it follows from \eqref{eq:conv_mmt} that for any fixed $n$ (recall that $\sL_Nb_N$ depends on $n$),
\begin{equation}\label{eq:LNbN_limit}
\lim_{N\to\infty}\bE\Big[(\sL_Nb_N)^T(\sL_N b_N)\Big]=\sum_{k=1}^n\bE_{(1,0)}\big[m_k^2\big].
\end{equation}
By \cite[Remark 4.22]{li.song.ea:26:on}, 
\begin{equation}\label{eq:parseval_local_time}
\sum_{k=1}^\infty m_k^2
=(\kappa_\alpha\hat\beta)^2
\int_0^1\int_0^1|r-s|^{2H-2}\dd L_r\dd L_s,
\quad \bP_{(1,0)}\text{-a.s.},
\end{equation}
and it is $L^1(\bP_{(1,0)})$-integrable. Hence, for any fixed $n$,
\begin{equation}\label{eq:n_covariance_remainder}
\lim_{N\to\infty}
\bE\Big[\var_{\bbP}(Y_N)-(\sL_N b_N)^T(\sL_N b_N)\Big]
=\sum_{k=n+1}^\infty\bE_{(1,0)}\big[m_k^2\big].
\end{equation}

Given $\varepsilon>0$, since $\sum_{k=1}^{\infty}\bE_{(1,0)}[m_k^2]<\infty$,  by \eqref{eq:n_covariance_remainder}, we can choose $n_0=n_0(\gep)$ and $N_0=N_0(n_0)$, such that for $n=n_0$ and any $N\geq N_0$,
\begin{equation*}
\bE\Big[\var_{\bbP}(Y_N)-(\sL_Nb_N)^T(\sL_Nb_N)\Big]\leq\varepsilon.
\end{equation*}
Therefore,
\begin{equation}\label{eq:large-N-n0}
\Big\|
\what{Z}_N^{(K)}-\what{Z}_{N,n_0}^{(K)}\Big\|_{L^2(\bbP)}^2\leq e^K\varepsilon,
\quad\text{for any}~N\geq N_0.
\end{equation}

Note that for every $n\geq n_0$,
\begin{equation*}
\sigma(\xi_{1,N},\ldots,\xi_{n_0,N})\subset\sigma(\xi_{1,N},\ldots,\xi_{n,N}),
\end{equation*}
and
\begin{equation*}
\what{Z}_N^{(K)}
-\what{Z}_{N,n_0}^{(K)}
=
\what{Z}_N^{(K)}
-\what{Z}_{N,n}^{(K)}
+
\what{Z}_{N,n}^{(K)}
-\what{Z}_{N,n_0}^{(K)}.
\end{equation*}
Since $\what{Z}_{N,n}^{(K)}
-\what{Z}_{N,n_0}^{(K)}$ is $\sigma(\xi_{1,N},\ldots,\xi_{n,N})$-measurable
and 
$\bbE[\cdot|\xi_{1,N},\ldots,\xi_{n,N}]$
is the orthogonal projection onto the space generated by $\sigma(\xi_{1,N},\ldots,\xi_{n,N})$, it follows that by Pythagorean theorem. for all $n\ge n_0$, 
\begin{equation*}
\Big\|
\what{Z}_N^{(K)}
-\what{Z}_{N,n_0}^{(K)}
\Big\|_{L^2(\bbP)}^2
=
\Big\|
\what{Z}_N^{(K)}
-\what{Z}_{N,n}^{(K)}
\Big\|_{L^2(\bbP)}^2+\Big\|
\what{Z}_{N,n}^{(K)}
-\what{Z}_{N,n_0}^{(K)}
\Big\|_{L^2(\bbP)}^2\ge \Big\|
\what{Z}_N^{(K)}
-\what{Z}_{N,n}^{(K)}
\Big\|_{L^2(\bbP)}^2.
\end{equation*}
Combining this with \eqref{eq:large-N-n0}, we obtain that
\begin{equation}\label{eq:large-n_covariance_remainder}
\Big\|
\what{Z}_N^{(K)}
-\what{Z}_{N,n}^{(K)}
\Big\|_{L^2(\bbP)}^2
\leq
e^K\varepsilon,\quad\text{for any}~n\geq n_0~\text{and}~N\geq N_0.
\end{equation}

To conclude \eqref{eq:sup_L2_diff}, it remains to treat the case  when $N<N_0$ and $n$ is sufficiently large. It suffices to show that for any fixed $N$, there exists $n_N$, such that for any $n\geq n_N$,
\begin{equation}\label{eq:no_diff}
\what{Z}_{N,n}^{(K)}:=\bE\Big[\what{Z}_N^{(K)}\Big|\xi_{1,N},\cdots,\xi_{n,N}\Big]=\what{Z}_N^{(K)}.
\end{equation}
Indeed, once \eqref{eq:no_diff} is proved, for $1\leq N\leq N_0-1$, correspondingly, we have $n_1,\cdots,n_{N_0-1}$, and we set $\bar{n}:=\max\{n_1,\cdots,n_{N_0-1}\}$. Then we have that
\begin{equation}\label{eq:small_N}
\Big\|
\what{Z}_N^{(K)}
-\what{Z}_{N,n}^{(K)}
\Big\|_{L^2(\bbP)}^2
=0,\quad\text{for any}~n\geq\bar{n}~\text{and}~N< N_0.
\end{equation}
This, combined with \eqref{eq:large-n_covariance_remainder}, yields
\begin{equation}\label{eq:final_bound}
\sup\limits_{N\geq1}\Big\|
\what{Z}_N^{(K)}
-\what{Z}_{N,n}^{(K)}
\Big\|_{L^2(\bbP)}^2
\leq e^K\gep,\quad\text{for any}~n\geq(\bar{n}\vee n_0),
\end{equation}
which completes the proof.

Now we prove that for each fixed $N$,  \eqref{eq:no_diff} holds for sufficiently large $n$.   Clearly, $\what{Z}_N^{(K)}$ is $\sigma(\omega_1,\cdots,\omega_N)$-measurable and $\sigma(\xi_{1,N},\cdots,\xi_{n,N})\subset\sigma(\omega_1,\cdots,\omega_N)$ by \eqref{def:xi_kN}. Hence, to prove \eqref{eq:no_diff}, it suffices to show that $\sigma(\omega_1,\cdots,\omega_N)\subset\sigma(\xi_{1,N},\cdots,\xi_{n,N})$ for sufficiently large $n$. 

 By  Remark~\ref{extension-f},  there exists a positive constant $C_{H,N}$ such that 
$|I_{j,N}(f)|\leq C_{H,N}\|f\|_{\cH},
1\leq j\leq N,$
so each $I_{j,N}$ is a continuous linear functional on $\cH$. Hence,
\[
T_N(f):=\big(I_{1,N}(f),\ldots,I_{N,N}(f)\big)
\]
defines a continuous linear mapping from $\cH$ to $\bbR^N$. On the other hand, it is not hard to see that $T_N$ is onto. Indeed, for
any $\boldsymbol y=(y_1,\ldots,y_N)\in\bbR^N$, let $f_{\boldsymbol y}$
be the step function which is equal to $y_j$ on
$[1-j/N,1-(j-1)/N)$. Then $f_{\boldsymbol y}\in\cH$ and
$I_{j,N}(f_{\boldsymbol y})
=N\int_{(j-1)/N}^{j/N}f_{\boldsymbol y}(1-s)\dd s
=y_j$, for $1\leq j\leq N$.

Recall that $\{e_k\}_{k\geq1}$ is an orthonormal basis of $\cH$. Hence, for any
$f\in\cH$, there exists a sequence $\{f_m\}_{m\geq1}$ with
$f_m\in\operatorname{span}\{e_k:k\geq1\}$ such that $f_m\to f$ in $\cH$ as $m\to\infty$. By the continuity of $T_N$, $\lim_{m\to\infty}T_N(f_m)=T_N(f)$. By the linearity of $T_N$, $T_N(f_m)\in\operatorname{span}\{(I_{1,N}(e_k),\ldots,I_{N,N}(e_k)):k\geq1\}$ for all $m\geq1$. Therefore, we obtain that
\[
T_N(\cH)=\overline{
\operatorname{span}\big\{(I_{1,N}(e_k),\ldots,I_{N,N}(e_k)):k\geq1\big\}}=\bbR^N.
\]

Since $\bbR^N$ is a finite-dimensional space, any linear subspace of $\bbR^N$ is closed. Hence,
\[
\operatorname{span}\big\{(I_{1,N}(e_k),\ldots,I_{N,N}(e_k)):k\geq1\big\}=\bbR^N,
\]
Moreover, there exist finitely many indices $k_1,\ldots,k_N$, such that
\begin{equation*}
\operatorname{span}\big\{
(I_{1,N}(e_{k_i}),\ldots,I_{N,N}(e_{k_i})): 1\leq i\leq N\big\}=\bbR^N.
\end{equation*}
Let $n_N:=\max\{k_1,\ldots,k_r\}$. Then
\begin{equation*}
\operatorname{span}\big\{(I_{1,N}(e_k),\ldots,I_{N,N}(e_k)): 1\leq k\leq n_N\big\}=\bbR^N.
\end{equation*}

Now for each $j=1,\ldots,N$, there exist constants
$a_{j,1},\ldots,a_{j,n_N}$ such that
\begin{equation*}
\sum_{k=1}^{n_N}
a_{j,k}
\big(I_{1,N}(e_k),\ldots,I_{N,N}(e_k)\big)
=\mathbf e_j,
\end{equation*}
where $\mathbf e_j=(0,\cdots,0,1,0,\cdots,0)$ denotes the $j$-th standard basis vector of $\bbR^N$. Hence, by \eqref{def:xi_kN}, 
\begin{align*}
\sum_{k=1}^{n_N}a_{j,k}\xi_{k,N}
=
N^{-H}
\sum_{k=1}^{n_N}a_{j,k}
\sum_{\ell=1}^N I_{\ell,N}(e_k)\omega_\ell=
N^{-H}
\sum_{\ell=1}^N
\left(
\sum_{k=1}^{n_N}a_{j,k}I_{\ell,N}(e_k)
\right)\omega_\ell=
N^{-H}\omega_j.
\end{align*}
which shows $\operatorname{span}\{\omega_1,\ldots,\omega_N\}\subset
\operatorname{span}\{\xi_{1,N},\ldots,\xi_{n_N,N}\}$ and thus \eqref{eq:no_diff} holds for $n\ge n_N$. This completes the proof.
\end{proof}

\begin{remark}
One may choose a particular orthonormal basis for which the above $n_N$ can be taken as $N$. More precisely, let $\{e_k\}_{k\geq1}$ be obtained by applying the
Gram--Schmidt procedure to polynomials $\{s^\ell\}_{\ell\geq0}$.
Then for $0\leq \ell\leq N-1$ and $1\leq j\leq N$, we have
\[
I_{j,N}(s^\ell)
=
N\int_{(j-1)/N}^{j/N}(1-s)^\ell\dd s
=
\frac{(N-j+1)^{\ell+1}-(N-j)^{\ell+1}}
{(\ell+1)N^\ell}.
\]
We now check that the vectors 
\[
\big(I_{1,N}(s^\ell),\ldots,I_{N,N}(s^\ell)\big),
\qquad 0\leq \ell\leq N-1,
\]
are linearly independent, and therefore form a basis of $\bbR^N$. Note that
\[
\big(I_{1,N}(s^\ell),\ldots,I_{N,N}(s^\ell)\big)=\frac{1}{(\ell+1)N^l}\big(B_{\ell,1},\cdots,B_{\ell,N})\quad\text{with}~B_{\ell,j}=(N-j+1)^{\ell+1}-(N-j)^{\ell+1}.
\]
Then it suffices to show that $B_\ell:=(B_{\ell,1},\cdots,B_{\ell,N})$ for $0\leq\ell\leq N-1$ are linearly independent. 

Suppose that there are constants $c_0,\cdots,c_{N-1}$ such that
\[
\sum\limits_{\ell=0}^{N-1}c_\ell B_{\ell,j}=0, \quad\text{for all}~1\leq j\leq N.
\]
Then the polynomial $P_N(x)=\sum_{\ell=0}^{N-1}c_\ell \big((x+1)^{\ell+1}-x^{\ell+1}\big)$ has $N$ roots  $0,1,\cdots,N-1$, which implies $P_N(x)\equiv0$ by $\operatorname{deg}(P_N)\leq N-1$. Hence, $c_0=\cdots=c_{N-1}=0$.

Since the Gram--Schmidt procedure is an invertible
transform, it follows that
\[
\operatorname{Span}\left\{
\big(I_{1,N}(e_k),\ldots,I_{N,N}(e_k)\big):
1\leq k\leq N
\right\}
=
\bbR^N.
\]
Consequently,
\[
\sigma(\xi_{1,N},\ldots,\xi_{n,N})
=
\sigma(\omega_1,\ldots,\omega_N),
\quad\text{for any}~n\geq N.
\]
\end{remark}
\subsection*{Acknowledgment}
J.\ Song is partially supported by NSFC (No.\ 12471142) and the Fundamental Research Funds for the Central Universities. R.\ Wei is supported by NSFC (No.\ 12401170) and Xi’an Jiaotong-Liverpool University Research Development Fund RDF-23-01-024.

\bibliographystyle{plain}
\bibliography{project}
\end{document}